\documentclass[twoside,11pt]{amsart}
\usepackage{mathrsfs}
\usepackage[dvips]{color}

\usepackage{amsmath}
\usepackage{amsfonts}
\usepackage{amssymb}
\usepackage{latexsym}
\usepackage{epsfig}

\usepackage{float}
\restylefloat{figure}
\usepackage{subfigure}

\usepackage{epic}

\usepackage{graphicx}
\usepackage{graphics}
\usepackage{moreverb}
\usepackage{multicol}

\usepackage{enumerate}

\newtheorem{theorem}{Theorem}[section]
\newtheorem{lemma}{Lemma}[section]

\newtheorem{nota}{Remark}[section]

\newtheorem{definition}{Definition}[section]
\newtheorem{example}{Example}[section]

\newcommand{\disp}{\displaystyle}

\newcommand{\barint}{\hbox{$\int$\kern-0.75\intwidth
\vrule width 0.5\intwidth height 2.4pt depth
-2pt\kern0.25\intwidth}}
\newlength\intwidth
\setbox0=\hbox{$\int$} \intwidth=\wd0

\newcommand\avint{\hbox{\hbox{$\displaystyle \int$}\hbox{\kern-.9em{$-$}}}}

\newcommand\smavint{\hbox{\hbox{$\int$}\hbox{\kern-.75em{$-$}}}}

\newcommand{\arginf}{\mathop{{\rm arg\,inf}}}

\def\XXint#1#2#3{{\setbox0=\hbox{$#1{#2#3}{\int}$}
\vcenter{\hbox{$#2#3$}}\kern-.51\wd0}}

\newcommand{\be}{\begin{eqnarray}}
\newcommand{\ee}{\end{eqnarray}}

\newcommand{\ben}{\begin{eqnarray*}}
\newcommand{\een}{\end{eqnarray*}}

\newcommand{\mass}{R}
\newcommand{\moment}{M}

\numberwithin{equation}{section}

\begin{document}

\title[Weak solutions of pressureless Euler-Poisson equations] {A variational representation of weak solutions\\for the
pressureless Euler-Poisson equations}

\author[Eitan Tadmor]{Eitan Tadmor}
\address[Eitan Tadmor]{\newline
        Department of Mathematics,  Institute for Physical Science and Technology\newline
        and Center of Scientific Computation And Mathematical Modeling (CSCAMM)\newline
        University of Maryland\newline
        College Park, MD 20742 USA}
 \email[]{tadmor@cscamm.umd.edu}
\urladdr{http://www.cscamm.umd.edu/\~{}tadmor}

\author[Dongming Wei]{Dongming Wei}
\address[Dongming Wei]{\newline
    Department of Mathematics\newline
        University of Wisconsin\newline
        Madison, WI 53706 USA}
\email []{dwei@math.wisc.edu}

\date\today


\thanks{\textbf{Acknowledgment.} Research was supported in part by NSF grants DMS07-07949 and DMS10-08397 (ET), FRG07-57227(ET and DW) and 07-57285(DW)}


\begin{abstract}

We derive an explicit  formula for  global weak
solutions of the one dimensional system of pressure-less Euler-Poisson equations.
Our  variational formulation  is an  extension of the well-known  formula for  entropy solutions of the scalar inviscid Burgers' equation:
since the characteristics of the Euler-Poisson
equations are parabolas, the representation of their weak solution takes the form of a
``quadratic" version of the celebrated Lax-Oleinik variational formula.
Three cases are considered. (i) The variational formula recovers the ``sticky particle" solution in the \emph{attractive case}; (ii) It represents  a \emph{repulsive solution}  which is different than the one obtained by the sticky particle construction; and (iii)   the result is further extended to the  \emph{multi-dimensional}  Euler-Poisson system with radial symmetry.

\end{abstract}
\maketitle

\section{Introduction}\label{Intro}
\setcounter{figure}{0} \setcounter{equation}{0}

We study the system of \emph{pressureless Euler-Poisson} equations
\begin{subequations}\label{EP1}
\begin{eqnarray}
\rho_t+(\rho u)_x&=&0, \label{EP-mass}\\
(\rho u)_t+(\rho u^2)_x&=&\kappa\rho E, \quad E_x=\rho, \quad
E(-\infty,t)=0. \label{EP-moment}
\end{eqnarray}
\end{subequations}
Here $E$ is the electric field, $\kappa$ is a given physical
constant which signifies the type of the underlying forcing. We
distinguish between three different cases depending on the sign on
$\kappa$.

\noindent (i) The attractive case, $\kappa<0$. Solutions of
(\ref{EP1}) always breakdown at a finite time, $t=t_c$, where
$u_x(\cdot,t\uparrow t_c) \rightarrow -\infty$. 

\noindent (ii) The repulsive case, $\kappa>0$.  It was shown in
\cite{EnLT01} that in the presence of repulsing forcing,  there is a
large class of so-called \emph{sub-critical} initial data,
$u'_0>-\sqrt{2k\rho_0}$,  for which (\ref{EP1}) admits global smooth
solutions  governed by
\begin{equation}\label{PB}
u_t+uu_x=\kappa E;
\end{equation}
 see \cite{LT02, LT03, ChengT08, TWe08, ChaeT08, LTW10,
We10a, We10b} for  the critical threshold phenomena in related Euler and
Euler-Poisson systems. For super-critical data, however, the
repulsive Euler-Poisson solution breaks-down  at a critical time,
$t=t_c <\infty$, after which (\ref{PB}) is no longer equivalent to
(\ref{EP-moment}).

\noindent
(iii) Finally, there is the neutral case $\kappa=0$, governed by the \emph{pressureless Euler} equations,
\begin{subequations} \label{con}
\begin{equation}\label{con-mass}
\rho_t+(\rho u)_x=0
\end{equation}
\begin{equation}\label{con-moment}
(\rho u)_t+(\rho u^2)_x=0.
\end{equation}
\end{subequations}
For increasing data, $u_0'>0$, solutions of (\ref{con}) remains smooth (corresponding to sub-critical data in the  limiting case $\kappa=0$), and are governed by
 the decoupled inviscid Burgers equation
\begin{equation} \label{Burgers}
u_t+uu_x=0.
\end{equation}
It is well known that for general non-increasing initial data
solutions of (\ref{Burgers}) will lose their initial $C^1$
regularity  at a finite time \cite{La57}. Thereafter,
(\ref{Burgers})  and (\ref{con-moment})  are not equivalent.
Solutions of Burgers' equation (\ref{Burgers}) past the critical
time develop shock discontinuities and are given by the celebrated
Lax-Oleinik formula \cite{Ev98}. Solutions of the pressureless Euler
system (\ref{con}) develop $\delta$-shocks. Their construction
attracted  great attention in the 90's: they were obtained using
sticky particles formulation in \cite{Z70,CPY90,BG98,CKR07},  using
a generalized variational principle in \cite{ERS96}, and through a
vanishing pressure limiting process \cite{Bou93,CL03,CL04} and the
references therein. Uniqueness  was proved in \cite{HW01}.

Here, we are interested in the global weak solution of the (non-neutral)
Euler-Poisson system (\ref{EP1}),  in either  the attractive case,
$\kappa <0$,  or in the repulsive case $\kappa >0$ subject to
general  initial data, beyond  the global regularity in the sub-critical regime studied in \cite{EnLT01}. Among the  few known results we mention the existence and uniqueness
result of global weak solutions for
 the pressureless Euler-Poisson system \cite{NTu08}.

In this paper, we construct an explicit formula for weak solutions of the Euler-Poisson (\ref{EP1}).
For the inviscid Burgers' equation, generic shock develops
due to the intersection of \emph{straight} characteristics, after which entropy
solutions are given by the variational Lax-Oleinik formula. For the pressureless Euler-Poisson
equations, characteristics become \emph{quadratic}, and the solution, $\rho(x,t), \rho u(x,t)$ will be expressed in terms of the minimizer, $y(x,t)$ of the weighted quadratic form
\begin{equation}\label{EPysmu}
 y(x,t) =\sup\limits_{y}\bigg\{y \, \Big|
\, y = \arginf_y \int_0^y Q_{x,t}(s)\rho_0(s)ds \bigg \},   \quad
 Q_{x,t}(s):=s+tu_0(s)+\frac12 \kappa E_0(s)t^2-x.
\end{equation}

\begin{theorem}\label{IntroThm:1Ds}
[$L^1$ initial density] Consider 
the pressureless Euler-Poisson system \eqref{EP1} subject to initial
data $u_0(x)\!:=\!u(x,0)\in C^1(\mathbb{R})$ and
 $0\leq \rho_0(x)\!:=\!\rho(x,0)\in L^1(\mathbb{R})$.
Set $E_0(s):=\displaystyle\int_{-\infty}^{s} \rho_0(w)dw$ as the corresponding initial electric field. Then,
$(\rho,\rho u)=(\partial_x \mass,\partial_x \moment)$  is a weak
solution of \eqref{EP1}, where $\mass(x,t)\equiv\mass(y(x,t))$ and
$\moment(x,t)\equiv\moment(y(x,t))$ are given by

\begin{equation}
\mass(x,t)=\int_0^{y(x,t)} \rho_0(s) ds, \qquad \moment(x,t)=\int_0^{y(x,t)} \Big(u_0(s)+\kappa
E_0(s)t\Big)\rho_0(s)  ds.\label{WSEPsmu}
\end{equation}
\end{theorem}

\begin{nota}
The presence of Poisson potential is responsible for `converting'
the straight characteristics familiar from the Burgers' equation,
into parabola in Euler-Poisson equations. This is reflected in the
variational formula (\ref{WSEPsmu}) through the additional term
$\frac12 \kappa E_0(s)t^2$ in (\ref{EPysmu}).
\end{nota}

\begin{nota}
The representation formula \eqref{WSEPsmu} applies for all
$\kappa\in \mathbb{R}$. When $\kappa=0$, one recovers the
variational formulation of \cite{HW01} for the  pressureless Euler
equations. In the particular case of $\kappa=0$ and $\rho_0\equiv 1$, \eqref{EPysmu},\eqref{WSEPsmu}  become
\[
y(x,t) :=\sup\limits_{y}\bigg\{y \, \Big|
\, y = \arginf_y \int_0^y tu_0(s)ds+ \frac{(y-x)^2}{2}-\frac{x^2}{2}\bigg \},
\]
recovering the celebrated Lax-Oleinik formula.
\end{nota}

\begin{nota}
For $\kappa\leq 0$, our representation formula \eqref{WSEPsmu} gives
the ``sticky particle" solution of the pressureless Euler-Poisson
system. For $\kappa>0$, the weak solution given by (\ref{WSEPsmu})
is different from the one corresponding to the ``sticky particle"
model. We discuss the details in sections \ref{APEP} and \ref{RPEP}.
\end{nota}

We can extend the result of theorem \ref{IntroThm:1Ds} to more
general initial measure data, $\rho(x,0)\in \mathscr{M}_{+}(\mathbb{R})$ and
$m(x,0)=\big(\rho(x,0)u(x,0)\big)\in \mathscr{M}(\mathbb{R})$.

\begin{theorem}\label{IntroThm:1Dg}
[Measure initial density] Consider the pressureless Euler-Poisson
system (\ref{EP1}) subject to initial data,
$m(x,0)=\rho(x,0)u(x,0)\in \mathscr{M}(\mathbb{R})$ and
$\rho_0(x)\!:=\!\rho(x,0)\in \mathscr{M}_{+}(\mathbb{R})$, such that
$\int_{-\infty}^{\infty}\rho_0(x)dx<\infty$ and $u_0(x)$ is
piecewise continuous.
The  corresponding initial electric
field, $E_0$, is given  by the average rule
\begin{equation}\label{E0s}
E_0(s):=\frac12\bigg(\displaystyle\int_{-\infty}^{s-}\rho_0(w)dw+\displaystyle\int_{-\infty}^{s+}
\rho_0(w)dw\bigg),
\end{equation}
Then, $(\rho,\rho u)=(\partial_x \mass,\partial_x \moment)$ is a
weak solution of (\ref{EP1}), where $\mass\equiv\mass(y(x,t))$ and
$\moment\equiv\moment(y(x,t))$ are given by
\begin{subequations}\label{eqs:RM0s}
\begin{equation}
\mass(x,t)=\left\{\begin{array}{ll}\displaystyle \int_0^{y(x,t)+} \rho_0(s) ds, & \ {\rm if} \ Q_{x,t}(y(x,t))\leq0,\\ \\
\displaystyle \int_0^{y(x,t)-} \rho_0(s) ds, & \ {\rm if} \
Q_{x,t}(y(x,t))> 0, \end{array}\right.
\end{equation}
and
\begin{equation}
\moment(x,t)= \left\{\begin{array}{ll}\displaystyle \int_0^{y(x,t)+}
\Big(u_0(s)+\kappa
E_0(s)t\Big)\rho_0(s)  ds, & \ {\rm if} \ Q_{x,t}(y(x,t))\leq0,\\ \\
\displaystyle \int_0^{y(x,t)-} \Big(u_0(s)+\kappa
E_0(s)t\Big)\rho_0(s)  ds, & \ {\rm if} \ Q_{x,t}(y(x,t))> 0.
\end{array}\right.
\end{equation}
\end{subequations}
\end{theorem}
We note in passing that since   $y(\cdot,t)$ is monotonically increasing, consult lemma \ref{properQ} below, the one-sided limits in (\ref{eqs:RM0s}), $y(\cdot,t)\pm$, are well-defined.

These  results  can be extended to the \emph{weighted}
multi-dimensional Euler/Euler-Poisson systems with symmetry, which
is the content of our third main theorem.

\begin{theorem}\label{IntroThm:MEP}
Consider the $n$-dimensional weighted Euler-Poisson equations
\begin{subequations}\label{MultiDEP}
\begin{equation}\label{MultiDrho}
\rho_t + \nabla \cdot(\rho \mathbf{u})=0, \quad
\mathbf{u}(\cdot,t): \mathbb{R}^n \mapsto \mathbb{R}^n,
\end{equation}
\begin{equation}\label{MultiDu}
(\rho\mathbf{ u})_t+\nabla\cdot(\rho\mathbf{u} \otimes \mathbf{u})=
\kappa\rho|\mathbf{x}|^{n-1}\nabla V, \quad \Delta V=\rho,
\end{equation}
\end{subequations}
subject to spherically symmetric initial data,
$\displaystyle \rho_0(\mathbf{x})\mathbf{u}_0(\mathbf{x})\!:=\!
\rho_0(|\mathbf{x}|)u_0(|\mathbf{x}|)\frac{\mathbf{x}}{|\mathbf{x}|}$,
such that\newline  $0\leq \rho(\mathbf{x},0)=\rho_0(|\mathbf{x}|)\in
L^1(\mathbb{R}^n)$, 
and ${u}_0(|\mathbf{x}|)\in C^1(\mathbb{R}_+)$.
Then the system (\ref{MultiDEP}) admits a
 weak radial solution, $(\rho,\rho \mathbf{u})$, of the form
\begin{subequations}
\begin{equation}
\begin{aligned}
\rho(\mathbf{x},t)= \displaystyle\frac{1}{r^{n-1}}\frac{\partial
\mass(r,t)}{\partial r}, \qquad
\rho(\mathbf{x},t)\mathbf{u}(\mathbf{x},t)=\displaystyle\frac{\mathbf{x}}{r}\frac{1}{r^{n-1}}\frac{\partial
\moment(r,t)}{\partial r},  \qquad r:=|\mathbf{x}|.
\end{aligned}
\end{equation}
Here, $\mass(r,t)\equiv\mass(y(r,t))$ and $\moment(r,t)\equiv\moment(y(r,t))$ are given by,
\begin{align}
\mass(r,t)=\displaystyle\int_0^{y(r,t)} \rho_0(s)s^{n-1} ds, \qquad
\moment(r,t)&=\displaystyle\int_0^{y(r,t)}
\rho_0(s)\Big(w_0(s)+\kappa E_0(s)t\Big)s^{n-1} ds,
\end{align}
where  $E_0$ is the initial electric field,
\begin{equation}
\displaystyle E_0(r):= r^{n-1}V_r(|{\mathbf x}|,0)=  \int_0^{r} s^{n-1}\rho_0(s)ds, \quad r=|\mathbf{x}|,
\end{equation}
and $y=y(r,t)$ is determined in terms of the quadratic form $Q_{r,t}(s)=s+tu_0(s)+\frac12 \kappa E_0(s)t^2-r$,
\begin{equation}
 y(r,t) :=\sup\limits_{y}\bigg\{y \, \Big|
\, y = \arginf_y \int_0^y Q_{r,t}(s)s^{n-1}\rho_0(s)ds \bigg \}.
\end{equation}
\end{subequations}
\end{theorem}

\begin{nota}
If $\kappa=0$, theorem \ref{IntroThm:MEP} yields the  formula for  a
global weak solution of the multi-dimensional pressureless Euler
system with symmetry. If $n=1$, it recovers the formula for the 1D
pressureless Euler-Poisson system, that is, theorem
\ref{IntroThm:1Ds} with a anti-symmetric $E_0$.
\end{nota}

The paper is organized as follows. In section \ref{APEP}, we
develop the formula for the attractive pressureless Euler-Poisson
system. In section \ref{RPEP}, we extend the formula to the
repulsive system, and explain different physical meanings of the
weak solutions for $\kappa<0$ and $\kappa>0$. In section \ref{MS},
we extend the formula to the multidimensional weighted Euler-Poisson system. 

\section{Weak solutions of the attractive pressureless Euler-Poisson system}\label{APEP}
\setcounter{figure}{0} \setcounter{equation}{0}

\subsection{From continuum to particles} \label{DefEF}

We begin with the transport of the center of mass in the smooth case.

\begin{lemma}\label{lemma:TCG}
Assume that $(\rho(\cdot,t),\rho u(\cdot,t))$ is a smooth solution
of the pressureless Euler-Poisson system (\ref{EP1})  for
$t\in[0,T]$. Let $x(\alpha,t)$ denote the particle path emanating
from $\alpha\in \mathbb{R}$. Then the center of mass of particles
emanating from the interval $[a,b]$ is given by the
parabola
\begin{equation}
X(t)=X_0+U_0 t+\displaystyle\frac12 \kappa \widetilde{E_0}t^2, \qquad
t\in[0,T]
\end{equation}
where
\begin{equation}
\displaystyle X_0=\frac{\displaystyle\int_a^b x\rho_0(x)dx}{\displaystyle\int_a^b\rho_0(x)dx},
\quad U_0=\displaystyle\frac{\displaystyle\int_a^b
\rho_0(x)u_0(x)dx}{\displaystyle\int_a^b\rho_0(x)dx}, \quad \widetilde{E_0}=\frac12
\Big(\int_{-\infty}^a \rho_0(x)dx+\int_{-\infty}^b\rho_0(x)dx\Big).
\end{equation}
\end{lemma}

\begin{proof}
We use the method of characteristics to obtain an explicit solution
of (\ref{EP1}), \cite{EnLT01}. Along the particle trajectory, the
equations of $x$ and $u$ are
\begin{subequations}
\begin{equation}\label{clx}
\displaystyle\frac{dx(\alpha,t)}{dt}=u(x(\alpha,t),t), \quad
x(\alpha,0)=\alpha,
\end{equation}
and
\begin{equation}\label{clu}
\displaystyle\frac{du(x(\alpha,t),t)}{dt}=\kappa
V_x(x(\alpha,t),t)=\kappa E(x(\alpha,t),t),
\end{equation}
\end{subequations}
where
\begin{equation}
E(x(\alpha,t),t)=\int_{-\infty}^{x(\alpha,t)} \rho(\xi)d\xi.
\end{equation}
Since
\[ \left. \begin{array}{lll}
\displaystyle\frac{d}{dt} E\Big(x(\alpha,t),t\Big) & = &
\displaystyle\frac{d}{dt}\Big(x(\alpha,t)\Big)\cdot
\rho\Big(x(\alpha,t),t) \Big) +
\displaystyle\int_{-\infty}^{x(\alpha,t)} \rho_t(\xi,t)d\xi \vspace{.1in}\\
& = & u(\alpha,t)\rho\Big(x(\alpha,t),t) \Big) - \displaystyle\int_{-\infty}^{x(\alpha,t)}
\Big(\rho(\xi,t)u(\xi,t)\Big)_{\xi}d\xi =0,\\
\end{array} \right. \]
the electric field remains constant along $x(\alpha,t)$. Let
\[ E_0(\alpha)=\int_{-\infty}^{x(\alpha,0)} \rho_0(s)ds, \quad
u_0=u(x(\alpha,0),0), \quad \rho_0=\rho(x(\alpha,0),0). \]
Then
(\ref{clu}) can be simplified to
\[
\displaystyle\frac{du(x(\alpha,t),t)}{dt}=\kappa E_0,
\]
which yields
\begin{equation}\label{solu-u}
u(x(\alpha,t),t)=u_0+\kappa E_0t.
\end{equation}
This together with (\ref{clx}) yield the equation of the particle
path
\begin{equation}\label{solu-x}
x(\alpha,t)=\alpha+u_0t+\displaystyle\frac{\kappa E_0 t^2}{2}.
\end{equation}
Let
\[
\Gamma(\alpha,t):=\displaystyle\frac{\partial x}{\partial
\alpha}=1+u_0't+\displaystyle\frac{\kappa \rho_0 t^2}{2}, \qquad
u_0':=\displaystyle\frac{\partial u_0(\alpha)}{\partial \alpha}
\]
Taking the $x$ derivative of $u(x(\alpha,t),t)$ yields
\begin{equation}\label{clux}
u_x(x(\alpha,t),t)=\displaystyle\frac{\partial u}{\partial
\alpha}/\displaystyle\frac{\partial x}{\partial
\alpha}=\displaystyle\frac{u_0'+\kappa\rho_0
t}{1+u_0't+\displaystyle\frac{\kappa\rho_0
t^2}{2}}=\displaystyle\frac{\Gamma_t(\alpha,t)}{\Gamma(\alpha,t)}.
\end{equation}
Plugging (\ref{clux}) into the mass equation (\ref{EP-mass}) yields
\[
\displaystyle\frac{d}{dt}\rho(x(\alpha,t),t)=
-u_x\rho=-\displaystyle\frac{\Gamma_t(\alpha,t)}{\Gamma(\alpha,t)}\rho(x(\alpha,t),t).
\]
Solving this equation, we obtain
\begin{equation}\label{solu-rho}
\rho(x(\alpha,t),t)=\displaystyle\frac{\rho_0}{\Gamma(\alpha,t)}.
\end{equation}
Equipped with (\ref{solu-u}), (\ref{solu-x}) and (\ref{solu-rho}),
we can find the position $X(t)$ of the center of gravity of the mass
on $[x(a,t),x(b,t)]$ at time $t$
\begin{eqnarray*}
 X(t) & = &
\displaystyle\frac{\displaystyle\int_{x(a,t)}^{x(b,t)} \xi \rho(\xi,t)
d\xi}{ \displaystyle\int_{x(a,t)}^{x(b,t)} \rho(\xi,t) d\xi}
 =  \displaystyle\frac{\displaystyle\int_{a}^{b} x(\alpha,t)
\rho(x(\alpha,t),t) \Gamma(\alpha,t) d\alpha}{
E_0(b)-E_0(a)}  \\
& = & \displaystyle\frac{\displaystyle\int_{a}^{b}
\Big(\alpha+u_0(\alpha)t+\displaystyle\frac{\kappa E_0(\alpha)
t^2}{2}\Big) \rho_0(\alpha)
 d\alpha}{
E_0(b)-E_0(a)}
 =  X_0+U_0 t+ \displaystyle\frac{\kappa t^2
\displaystyle\int_{a}^{b}
 E_0(\alpha)  \rho_0(\alpha)
 d\alpha}{
2(E_0(b)-E_0(a))} \\
& = & X_0+U_0 t+ \displaystyle\frac{\kappa t^2
\displaystyle\int_{a}^{b}
 E_0(\alpha)
 dE_0(\alpha)}{
2(E_0(b)-E_0(a))}
 =  X_0+U_0 t+ \displaystyle\frac{\kappa t^2 (E_0(b)^2-E_0(a)^2)}{
4(E_0(b)-E_0(a))} \\
 &= &  X_0+U_0 t+ \displaystyle\frac{1}{2}\kappa \widetilde{E_0}t^2. \\
\end{eqnarray*}
\end{proof}

Lemma \ref{lemma:TCG} tells us that if  we replace the initial mass
along $[a,b]$  by a Dirac mass, $\int_a^b \rho_0(x) dx$, situated at
location $X_0$, give it velocity $U_0$, and apply to this particle an electrical field
 $\widetilde{E_0}$, then the trajectory
of this particle is the same as the trajectory of the center of mass
of the whole initial interval $[a,b]$. Motivated by this lemma, we continue to
deal with measure densities. In particular,  if the density $\rho(\cdot,t)$ at location $x$ has  a Dirac mass with strength
$\overline{m}(x)$, then we set the electric field at that point to be
\[
E(x,t):=\displaystyle\int_{-\infty}^{x-}
\rho(s,t)ds+\frac{\overline{m}(x)}{2};
\]
 otherwise, $E(x,t):=\displaystyle\int_{-\infty}^x \rho(s,t)ds$.
Combining both cases, we arrive at the following definition of the electric field, which is in agreement with (\ref{E0s}).
\begin{definition}\label{def:EPVx}[The electric field]
We define the electric field, $E(x,t)$, as
\begin{equation}\label{EPVx}
E(x,t):=\frac12\bigg(\displaystyle\int_{-\infty}^{x-} \rho(s,t)ds+
\displaystyle\int_{-\infty}^{x+} \rho(s,t)ds \bigg),
\end{equation}
\end{definition}

Equipped with definition \ref{def:EPVx}, we will show later that the conclusion of Lemma \ref{lemma:TCG} remains valid even
after collision takes place. Therefore, setting $E$ as in (\ref{EPVx}) guarantees that
the trajectory of the center of mass is independent of whether there
is a collision or not. This is a key point which explains the validity behind our approach.

\subsection{The dynamics of two Dirac masses} \label{CWeakS}

Given the density $\rho(\cdot,t)$ as a non-negative measure and the velocity $u(\cdot,t)$ which is uniformly bounded, we introduce the corresponding \emph{mass} and \emph{momentum} of the system, which play an important role throughout the paper,
\begin{equation}
\mass(x,t):=\int_{-\infty}^{x+} \rho(s,t) ds, \qquad
\moment(x,t):=\int_{-\infty}^{x+} \rho(s,t)u(s,t)ds
\end{equation}

To illustrate the  construction of a weak solution solely
from the physical principles, we start with the simplest example of
two particles governed by an attractive force, $\kappa\leq 0$.
We consider  two particles with masses $m_1,m_2$ at initial positions
$y_1(0) < y_2(0)$, and initial velocity $u_1(0), u_2(0)$,
respectively. Thus, the initial density and momentum consist of two Dirac masses
\begin{eqnarray*}
\rho_0(x)&=&m_1\delta(x-y_1(0))+m_2\delta(x-y_2(0)),\\
\rho_0(x)u_0(x)&=&m_1u_1(0)\delta(x-y_1(0))+m_2u_2(0)\delta(x-y_2(0)).
\end{eqnarray*}
If there is no collision, the electrical fields which are associated with the first and
second particles are, respectively, $E_1:=\frac{m_1}{2}$ and
$E_2:=m_1+\frac{m_2}{2}$. Hence the velocity and location of the
first particle are
\begin{equation}\label{uy1t}
u_1(t)=u_1(0)+\kappa\frac{m_1}{2}t, \qquad
y_1(t)=y_1(0)+u_1(0)t+\frac12\kappa\frac{m_1}{2}t^2.
\end{equation}
The velocity and location of the second particle are
\begin{equation}\label{uy2t}
u_2(t)=u_2(0)+\kappa\left(m_1+\frac{m_2}{2}\right)t, \qquad
y_2(t)=y_2(0)+u_2(0)t+\frac12\kappa\left(m_1+\frac{m_2}{2}\right)t^2,
\end{equation}
Therefore the velocity and location of the center of gravity of the
system are
\[
\left. \begin{array}{lll} u(t) & = &
\displaystyle\frac{m_1u_1(t)+m_2u_2(t)}{m_1+m_2} =
 \displaystyle\frac{m_1u_1(0)+m_2u_2(0)}{m_1+m_2}
+\frac12\kappa\displaystyle\frac{m_1^2+2m_1m_2+m_2^2}{m_1+m_2}t
\vspace{0.05in} \\ & = & u(0)+ \kappa
\displaystyle\frac{m_1+m_2}{2}t, \qquad \qquad u(0):=\displaystyle\frac{m_1u_1(0)+m_2u_2(0)}{m_1+m_2}, \\
\end{array} \right.
\]
and
\[
\left. \begin{array}{lll} y(t) & = &
\displaystyle\frac{m_1y_1(t)+m_2y_2(t)}{m_1+m_2} \vspace{.05in} \\
& = & \displaystyle\frac{m_1y_1(0)+m_2y_2(0)}{m_1+m_2}
+\displaystyle\frac{m_1u_1(0)+m_2u_2(0)}{m_1+m_2}t
+\frac14\kappa\displaystyle\frac{m_1^2+2m_1m_2+m_2^2}{m_1+m_2}t^2
\vspace{0.05in} \\
& = & \displaystyle
y(0)+u(0)t+\frac12\kappa\displaystyle\frac{m_1+m_2}{2}t^2, \qquad
\qquad y(0):=\displaystyle\frac{m_1y_1(0)+m_2y_2(0)}{m_1+m_2}.
\end{array} \right.
\]
If collision takes place at time $\bar{t}$, then a new particle will
be formed at $y(\bar{t})$ with  velocity $u(\bar{t})$: the motion of
this new particle is determined by
\begin{subequations}
\begin{equation}\label{newu}
u_{new}(t)=u(\bar{t})+\kappa\displaystyle\frac{m_1+m_2}{2}(t-\bar{t})=u(t),
\end{equation}
and
\begin{equation}\label{newy}
y_{new}(t)=y(\bar{t})+u(\bar{t})(t-\bar{t})+\frac12\kappa\displaystyle\frac{m_1+m_2}{2}t^2=y(t).
\end{equation}
\end{subequations}
Hence, when and where the collision occurs, it will
\emph{not} change the trajectory of the center of mass of this system.

The standard way to describe the state of this two-particle system at any time $t$, would be to  provide the following information:

\medskip\noindent
(i) if the two particles have not collided
before time $t$: provide the position and velocity of each particle at
time $t$;

\medskip\noindent
(ii) if the two particles collided at some time before $t$:
then they are ``stuck" as one particle thereafter, and we provide the position
and velocity of this new particle at time $t$.

\medskip
Alternatively, we can
provide a complete description of the two-particle system in terms of the
mass $\mass(x,t)$ and momentum $\moment(x,t)$. To this end, take the difference of (\ref{uy2t}) and (\ref{uy1t}), to find
\[
y_2(t)-y_1(t)=\frac12\kappa\left(\frac{m_1}{2}+\frac{m_2}{2}\right)t^2+(u_2(0)-u_1(0))t+(y_2(0)-y_1(0)).\]
Since the two particle system is attractive, $\kappa\leq 0$,
there is a positive $t_c>0$, such that,
\[
\left\{\begin{array}{ll}y_2(t)>y_1(t) &  0\leq t<t_c,\\
y_2(t)=y_1(t), & t=t_c,\\
y_2(t)<y_1(t), & t>t_c.
\end{array}\right.
\]
Then, there are four possibilities to consider, depending on the relative positions of
$q_j(t):=y_j+u_jt+\frac12\kappa\displaystyle E_j t^2, \ j=1,2$:

\medskip\noindent
(i)
if both $q_1(t)$ and $q_2(t)$ are to the left of $x: \ q_1(t), q_2(t) \leq x$, then $\mass(x,t)=m_1+m_2$ and $\moment(x,t)=m_1(u_1+\kappa E_1t)+m_2(u_2+\kappa E_2t)$;

\medskip\noindent
(ii)
if both $q_1(t)$ and $q_2(t)$ are to the right $x: \ q_1(t), q_2(t) > x$, then $\mass(x,t)=0$, and
$\moment(x,t)=0$;

\medskip\noindent
(iii) if $q_1(t) \leq x$ and $q_2(t) > x$, then $\mass(x,t)=m_1$, and
$\moment(x,t)=m_1(u_1+\kappa E_1 t)$;

\medskip\noindent
(iv) finally, if $q_1(t)> x$
and $q_2(t) < x$, this means that collision
occurred earlier, and there is a new particle with mass
$m=m_1+m_2$. The
velocity and location of this new particle at time $t$ are given by  (\ref{newu}), (\ref{newy}),
\[ u(t)=\displaystyle\frac
{m_1(u_1+\kappa E_1 t)+m_2(u_2+\kappa E_2t)} {m_1+m_2}, \qquad y(t)=\displaystyle\frac
{m_1q_1(t)+m_2q_2(t)} {m_1+m_2}.
\]
Therefore, in this case $\mass(x,t)$ and $\moment(x,t)$ are determined by the sign of the expression
\[
\displaystyle\frac{m_1q_1(t)+m_2q_2(t)}{m_1+m_2}-x
=  \displaystyle\frac{m_1(q_1(t)-x)+m_2(q_2(t)-x)}{m_1+m_2},
\]
namely
\[
\left\{\begin{array}{ll}
\left(\begin{array}{l}\mass(x,t)=m_1+m_2 \\ \moment(x,t) = m_1(u_1+\kappa E_1t)+m_2(u_2+\kappa E_2t)\end{array}\right), &
\ {\rm if} \ \ \displaystyle\frac{m_1(q_1(t)-x)+m_2(q_2(t)-x)}{m_1+m_2}\leq 0\\
 \mass(x,t)=\moment(x,t)\equiv 0, & \ {\rm otherwise}. \end{array}\right.
\]

Summarizing the four cases above we observe:

\medskip\noindent
(i) if $q_1(t) \leq x$, then no matter whether collision
happened or not, $m_1$ will contribute to the mass $\mass(x,t)$;

\medskip\noindent
(ii) if $q_1(t) >x$ then the contribution of $m_1$ to the mass
$\mass(x,t)$ depends on whether the second particle is slow enough,
that is, whether $m_1(q_1(t)-x)+m_2(q_2(t)-x)\leq0$.

\subsection{The dynamics of general mass distribution} \label{GWeakS}

In this subsection, we extend the definition of $\mass$ and
$\moment$ to the general attractive case. Suppose that there exists
$Y_1$ such that
\[
q_s(t):=s+u_0(s)t+\frac12 \kappa E(s,0)t^2\leq x, \quad \forall s\leq Y_1.
\]
Then, independently whether collision occurred or not, the particles emanating
from $y\leq Y_1$ will end up to the left side of $x$ at time $t$, and
therefore, the part of the mass $\displaystyle\int_{-\infty}^{Y_1}
\rho_0(s) ds$ will be on the left side of $x$ at time $t$.
If, on the other hand,
\[
q_s(t)=s+u_0(s)t+\frac12 \kappa E(s,0)t^2>x, \quad \forall s\in(Y_1,Y_2],
\]
then the position of the mass of this part, $\displaystyle\int_{Y_1}^{Y_2}
\rho_0(s) ds$, relative to  $x$, depends on
whether there is enough slow material which will collide with this
part. That is, if there exists $Y_3>Y_2$, such that
\begin{equation} \label{PassX}
\left. \begin{array}{ll}
\displaystyle\frac{\displaystyle\int_{Y_1}^{Y_2}
q_s(t)\rho_0(s) ds+\displaystyle\int_{Y_2}^{Y_3}
q_s(t)\rho_0(s) ds}
{\displaystyle\int_{Y_1}^{Y_2}\rho_0(s)ds+\int_{Y_2}^{Y_3}\rho_0(s)ds}
=
\displaystyle\frac{\displaystyle\int_{Y_1}^{Y_3}
q_s(t)\rho_0(s) ds}
{\displaystyle\int_{Y_1}^{Y_3}\rho_0(s)ds}\leq x, \\ \end{array}
\right.
\end{equation}
then the mass on $(Y_1,Y_2]$ will be end up to be on the left side of $x$. We can rewrite (\ref{PassX}) in the equivalent form
\[
\displaystyle\int_{Y_1}^{Y_3}Q_{x,t}(s) \rho_0(s) ds\leq 0, \qquad Q_{x,t}(s)=q_s(t)-x.
\]
This is similar to the case (iv) of the two particle system, where
$\displaystyle \int_{Y_1}^{Y_2}\rho_0(s)ds$ corresponds to the first particle, and
$\displaystyle \int_{Y_2}^{Y_3}\rho_0(s)ds$ corresponds to the second particle.
Continuing this process, we  find that the exact amount of mass which ends
on the left side of $x$, is given by
\begin{equation}\label{M1}
\mass(x,t):=\int_{-\infty}^{y(x,t)} \rho_0(s) ds ,
\end{equation}
where $y(x,t)$ is determined by \eqref{EPysmu},
\begin{align}
y(x,t) =\sup_{y}\bigg\{y \, \Big| \,
 y = \arginf_y \int_0^{y} Q_{x,t}(s)\rho_0(s) ds\bigg\}, \qquad Q_{x,t}(s)= q_s(t)-x.
\end{align}
The momentum then follows
\begin{equation}
\moment(x,t):=\int_{-\infty}^{y(x,t)} \rho_0(s)\Big( u_0(s)+\kappa
E_0(s)t\Big)ds.
\end{equation}

\subsection{Proof of theorem \ref{IntroThm:1Ds} for the attractive case} \label{MainThm}

As a preparation for the proof of
theorem \ref{IntroThm:1Ds} with $\kappa\leq 0$, we first
characterize  the entropy solution of \eqref{EP1} in terms of a one-sided
Lipschiz condition. This is the content of the following two lemmas.

\begin{lemma}\label{properQ}
Let $y(x,t)$ be the minimizer in (\ref{EPysmu}). Then $y(\cdot,t)$ is non-decreasing
\[
y(x_1,t)\leq y(x_2,t), \quad \forall x_1<x_2.
\]
In particular, $y(x+,t):=\lim_{z\rightarrow x+}y(z,t)$ is well-defined and
$y(x,t)=y(x+,t), \  \forall x$.
\end{lemma}
\begin{proof}
Assume there exists $x_1<x_2$ such that $y(x_1,t)>y(x_2,t)$. Since
$y(x_1,t)$ minimizes $\displaystyle\int_0^y
Q_{x_1,t}(s)\rho_0(s)ds$, we have $
\displaystyle\int_{y(x_2,t)}^{y(x_1,t)} Q_{x_1,t}(s)\rho_0(s)ds\leq
0$. It follows that
\[ \displaystyle\int_{y(x_2,t)}^{y(x_1,t)} Q_{x_2,t}(s)\rho_0(s)ds=
\int_{y(x_2,t)}^{y(x_1,t)} Q_{x_1,t}(s)\rho_0(s)ds+
\int_{y(x_2,t)}^{y(x_1,t)} (x_1-x_2)\rho_0(s)ds \leq 0.\] That is,
\[ \displaystyle\int_{0}^{y(x_1,t)} Q_{x_2,t}(s)\rho_0(s)ds\leq
\displaystyle\int_{0}^{y(x_2,t)} Q_{x_2,t}(s)\rho_0(s)ds, \quad
y(x_1,t)>y(x_2,t),\] which is a contradiction to the definition of
$y(x_2,t)$.

Next, assume there exists $x$ such that $y(x+,t)>y(x,t)$. Then
$\displaystyle\int_{y(x,t)}^{y(x+,t)} Q_{x,t}(s)\rho_0(s)ds > 0$.
Therefore, for $\epsilon>0$ small enough, we have
\[ \left.\begin{array}{ll}
& \displaystyle\int_{y(x,t)}^{y(x+\epsilon,t)}
Q_{x+\epsilon,t}(s)\rho_0(s)ds \vspace{0.05in}\\
= & \displaystyle\int_{y(x,t)}^{y(x+,t)}
Q_{x+\epsilon,t}(s)\rho_0(s)ds+\displaystyle\int_{y(x+,t)}^{y(x+\epsilon,t)}
Q_{x+\epsilon,t}(s)\rho_0(s)ds \vspace{0.05in} \\
= & \displaystyle\int_{y(x,t)}^{y(x+,t)}
Q_{x,t}(s)\rho_0(s)ds-\int_{y(x,t)}^{y(x+,t)}
\epsilon(s)\rho_0(s)ds+\displaystyle\int_{y(x+,t)}^{y(x+\epsilon,t)}
Q_{x+\epsilon,t}(s)\rho_0(s)ds>0.
\end{array}\right. \]
That is
\[ \displaystyle\int_{0}^{y(x,t)}
Q_{x+\epsilon,t}(s)\rho_0(s)ds<\displaystyle\int_{0}^{y(x+\epsilon,t)}
Q_{x+\epsilon,t}(s)\rho_0(s)ds, \] which is a contradiction to the
definition of $y(x+\epsilon,t)$.
\end{proof}

\begin{lemma}\label{entropy}
[One sided Lipschiz condition] Consider the attractive Euler-Poisson
system \eqref{EP1} with $\kappa\leq 0$.  If $y(x,t_0)>y(x-,t_0)$, then the values in the open interval $\gamma\in \big(y(x-,t_0),y(x,t_0)\big)$, cannot be reached by evolving \eqref{EP1} along particle path, namely, $\gamma\neq \text{a minimizer} \ y(x,t)$.
\end{lemma}

Thus, if  $y(x,t_0)>y(x-,t_0)$, then according to (\ref{WSEPsmu}),
$\rho(x,t_0)$ will be a Dirac mass and the lemma \ref{entropy} tells us that
once a Dirac mass is formed, it will never split.

\begin{proof}
By the definition (\ref{EPysmu}), to prove the Lemma, it is enough
to show $\forall\, \gamma\in\Big(y(x-,t_0),y(x,t_0)\Big)$,
$\forall\, z\in \mathbb{R}$, $\forall\, t>t_0$, there exists $w$
such that $\displaystyle\int_0^w Q_{z,t}(s)ds<\int_0^{\gamma}
Q_{z,t}(s)ds$, that is, $\gamma$ does not minimize
$\displaystyle\int_0^y Q_{z,t}(s)\rho_0(s)ds$. For every
$\gamma\in\Big(y(x-,t_0),y(x,t_0)\Big)$, we have
\begin{align}
\int_{\gamma}^{y(x,t_0)} Q_{x,t_0}(s)\rho_0(s)ds\leq 0 \leq
\int_{y(x-,t_0)}^{\gamma} Q_{x,t_0}(s)\rho_0(s)ds,
\label{ybxl}
\end{align}
 These inequalities  can be rewritten in terms of the quadratics $f_1(t)$ and $f_2(t)$,
\[
f_2(t_0)\leq x  \leq f_1(t_0), \qquad f_j(t):={a_j}t^2+{b_j}t+{c_j},
\]
where the coefficients of $f_j(t)$'s are given by in terms of
$\displaystyle m_1:=\int_{y(x-,t_0)}^{\gamma} \rho_0(s)ds$ and
$\displaystyle m_2:=\int_{\gamma}^{y(x,t_0)} \rho_0(s)ds$:
\[
a_1=\frac{1}{m_1}\int_{y(x-,t_0)}^{\gamma} \frac12\kappa
E_0(s)\rho_0(s)ds, \ b_1=\frac{1}{m_1}\int_{y(x-,t_0)}^{\gamma}
u_0(s)\rho_0(s)ds, \ c_1=\frac{1}{m_1}\int_{y(x-,t_0)}^{\gamma}
s\rho_0(s)ds,
\]
and
\[
a_2=\frac{1}{m_2}\int_{\gamma}^{y(x,t_0)} \frac12\kappa
E_0(s)\rho_0(s)ds, \ b_2=\frac{1}{m_2}\int_{\gamma}^{y(x,t_0)}
u_0(s)\rho_0(s)ds, \ c_2=\frac{1}{m_2}\int_{\gamma}^{y(x,t_0)}
s\rho_0(s)ds.
\]
Notice that $\displaystyle {a_2}<\kappa E_0(\gamma)<{a_1}<0$, and
$\displaystyle {c_2}>\gamma>{c_1}$. Setting $f_1(t)=f_2(t)$, we find
two solutions, one positive and one negative. Denote the positive
one by $t_c$, then
\[
\left\{\begin{array}{ll}
f_1(t)<f_2(t), & 0<t<t_c\\
f_1(t)>f_2(t), & t>t_c.
\end{array}\right.
\]
It follows that $t_0\geq t_c$. For every $t>t_0$ and $z\in
\mathbb{R}$ we have
\begin{align}
\int_{y(x-,t_0)}^{\gamma} Q_{z,t}(s)\rho_0(s)ds=m_1(a_1 t^2+b_1t+c_1-z)=m_1(f_1(t)-z), \nonumber \\
\int_{\gamma}^{y(x,t_0)} Q_{z,t}(s)\rho_0(s)ds=m_2(a_2
t^2+b_2t+c_2-z)=m_2(f_2(t)-z). \nonumber
\end{align}
If $f_1(t)>z$, then
\[
\int_0^{\gamma}Q_{z,t}(s)\rho_0(s)ds>\int_0^{y(x-,t_0)}Q_{z,t}(s)\rho_0(s)ds;
\]
if $f_1(t)\leq z$, then $f_2(t)<f_1(t)\leq z$, which implies
\[
\int_0^{\gamma}Q_{z,t}(s)\rho_0(s)ds>\int_0^{y(x,t_0)}Q_{z,t}(s)\rho_0(s)ds.
\]
Therefore $\gamma$ does not minimize $\displaystyle\int_0^y
Q_{z,t}(s)\rho_0(s)ds$.
\end{proof}

\begin{nota}\label{ychar}
It is straightforward  to verify that  for smooth initial data, $u_0(x)\in C(\mathbb{R})$,
we have  $Q_{x,t}(y(x,t))=0$, i.e.,
\[
y(x,t)+tu_0(y(x,t))+\frac{1}{2}\kappa E_0(y(x,t))t^2=x, \quad \forall x.
\]
\end{nota}

\begin{nota}\label{yxcont}
Consider two adjacent discontinuous points $a$ and $b$ of $y(x,t)$,
i.e., $y(a-,t)<y(a,t)$, $y(b-,t)<y(b,t)$ and $y(x-,t)=y(x,t)$ for
every $x\in(a,b)$. Then $R$ and $M$ are continuous on $(a,b)$ at
time $t$. Moreover, combining lemma \ref{properQ} and remark
\ref{ychar}, we find that the characteristics emanate from the interval
$\Big(y(a,t),y(b,t)\Big)$ at $t=0$,
\[
x(t)=x_0+u_0(x_0)t+\frac12\kappa E_0(x_0)t^2, 
\quad \forall x_0\in(y(a,t),y(b,t)),
\]
will not intersect before $t$.
\end{nota}

As a final preparation for the proof
of theorem \ref{IntroThm:1Ds},  we calculate the distributional derivative of jump discontinuities across curves over surfaces, which will be useful when dealing the
singular part of $\mass_t$ and $\moment_x$. We summarize this calculation in the following lemma.

\begin{lemma}\label{singularMeasure}
Consider an open region, $V\subset \mathbb{R}^2$, and a curve,
$C(t,x(t))$ in the $t-x$ plane which divides $V$ into two parts, $V^l$ and $V^r$, and assume that
a function $S(t,x)$ is smooth on either side of this curve $C$, with values $S^{l}$ on $V^l$ and $S^r$ on $V^r$.
The weak derivative of $S$ is given by,
\[
S_x=S^{l}_x+S^{r}_x+ (S^{r}-S^{l})\nu_2\eta_C, \qquad
 S_t=S^{l}_t+S^{r}_t+ (S^{r}-S^{l})\nu_1\eta_C.
\]
Here,
$\nu$ is the  outward unit normal vector of $V^l$ on
boundary $C$,
\[
\nu=(\nu_1,\nu_2):=\left(\frac{\dot{x}(t)}{\sqrt{\dot{x}^2(t)+1}},\frac{-1}{\sqrt{(\dot{x}^2(t)+1}}\right).
\]
and  $\eta_{C}(t,x)$ is a surface measure supported on the
curve $C$, satisfying
\[
 \int_{\mathbf{R}^{2}} \phi(t,x)\eta_{C}(t,x) dt dx
=\int_{a}^{b} \phi(t,x(t))\displaystyle
\sqrt{\dot{x}^2(t)+1} dt, \quad \forall \phi \in
C_0^{\infty}(\mathbb{R}^{2}).
\]
\end{lemma}

\begin{proof}
For every $\phi(x,t)\in C_0^{\infty}(V)$,
\[ \left. \begin{array}{ll}
& \displaystyle\int_{V} \phi_x(x,t) S(x,t) dtdx  =
\displaystyle\int_{V^l} \phi_x(x,t) S^l(x,t) dtdx + \int_{V^r}
\phi_x(x,t) S^r(x,t) dtdx \vspace{0.05in} \\
= & -\displaystyle\int_{V^l} \phi(x,t) S_x^l(x,t) dtdx+
\int_{C}\phi(t,x(t))S^l(t,x(t))\nu_2 dC -\displaystyle \int_{V^r}
\phi(x,t) S_x^r(x,t) dtdx \vspace{0.05in} \\ & -\displaystyle
\int_{C}\phi(t,x(t))S^r(t,x(t))\nu_2 dC  \qquad \big(\mbox{where} \ (dC=\sqrt{(dx/dt)^2+1})\big) \\
= & -\displaystyle\int_{V^l} \phi(x,t) S_x^l(x,t) dtdx-\displaystyle
\int_{V^r} \phi(x,t) S_x^r(x,t) dtdx -
\int_{C}\phi(t,x(t))\Big(S^r(t,x(t))-S^l(t,x(t))\Big)\nu_2 dC,
 \vspace{0.05in} \\
= & -\displaystyle\int_{V^l} \phi(x,t) S_x^l(x,t) dtdx-\displaystyle
\int_{V^r} \phi(x,t) S_x^r(x,t) dtdx -
\int_{V}\phi(t,x(t))\Big(S^r(t,x(t))-S^l(t,x(t))\Big)\nu_2 \eta_C
dtdx.
\end{array} \right.
\]
Hence $S_x=S^{l}_x+S^{r}_x+ (S^{r}-S^{l})\nu_2\eta_C$. Similarly,
\[ \left. \begin{array}{ll}
& \displaystyle\int_{V} \phi_t(x,t) S(x,t) dtdx  =
\displaystyle\int_{V^l} \phi_t(x,t) S^l(x,t) dtdx + \int_{V^r}
\phi_t(x,t) S^r(x,t) dtdx \vspace{0.05in} \\
= & -\displaystyle\int_{V^l} \phi(x,t) S_t^l(x,t) dtdx+
\int_{C}\phi(t,x(t))S^l(t,x(t))\nu_1 dC -\displaystyle \int_{V^r}
\phi(x,t) S_t^r(x,t) dtdx \vspace{0.05in} \\ & -\displaystyle
\int_{C}\phi(t,x(t))S^r(t,x(t))\nu_1 dC  \vspace{0.05in} \\
= & -\displaystyle\int_{V^l} \phi(x,t) S_t^l(x,t) dtdx-\displaystyle
\int_{V^r} \phi(x,t) S_t^r(x,t) dtdx -
\int_{C}\phi(t,x(t))\Big(S^r(t,x(t))-S^l(t,x(t))\Big)\nu_1 dC,
 \vspace{0.05in} \\
= & -\displaystyle\int_{V^l} \phi(x,t) S_t^l(x,t) dtdx-\displaystyle
\int_{V^r} \phi(x,t) S_t^r(x,t) dtdx -
\int_{V}\phi(t,x(t))\Big(S^r(t,x(t))-S^l(t,x(t))\Big)\nu_1 \eta_C
dtdx.
\end{array} \right.
\]
Hence $S_t=S^{l}_t+S^{r}_t+ (S^{r}-S^{l})\nu_1\eta_C$.
\end{proof}

Equipped with lemmas \ref{entropy} and \ref{singularMeasure}, we are now ready to complete  the proof of theorem \ref{IntroThm:1Ds} in the attractive case,  $\kappa\leq 0$.

\begin{proof} [Proof of theorem \ref{IntroThm:1Ds} with $\kappa <0$]

\noindent {\bf Step \#1} [the mass equation].
First, we show $\rho$ and $\rho u$ satisfy the mass
equation (\ref{EP-mass}) in the weak sense. We need to verify that
\begin{equation}\label{EPMPsmu}
\int_{0}^{\infty}\int_{-\infty}^{\infty} \rho_t \phi+(\rho u)_x\phi
dxdt =\int_{0}^{\infty}\int_{-\infty}^{\infty} \mass_{xt}
\phi+\moment_{xx}\phi dxdt =\int_{0}^{\infty}\int_{-\infty}^{\infty}
\mass\phi_{xt} +\moment\phi_{xx} dxdt
\end{equation}
vanishes for all test functions $\phi\in
C_c^{\infty}\Big([0,\infty)\times(-\infty,\infty)\Big)$. To this
end, we claim that
\begin{equation}\label{MPG3smu}
\displaystyle \frac{\partial \mass}{\partial t}=-\frac{\partial
\moment}{\partial x}.
\end{equation}
It follows (\ref{MPG3smu}) that there
exists a function $G$ such that
\begin{equation}\label{MPG4smu}
\mass(x,t)=-\displaystyle\frac{\partial G(x,t)}{\partial x}, \quad
\moment(x,t)=\displaystyle\frac{\partial G(x,t)}{\partial t}.
\end{equation}
By plugging (\ref{MPG4smu}) into (\ref{EPMPsmu}) we obtain
\begin{equation}\label{eq:wmasssmu}
\int_{0}^{\infty}\int_{-\infty}^{\infty} \rho_t \phi+(\rho u)_x\phi
dxdt=\int_{0}^{\infty}\int_{-\infty}^{\infty} G \phi_{xxt}-
G\phi_{xxt} dxdt =0, \quad \forall \phi\in
C_c^{\infty}\Big(\mathbb{R}\times \mathbb{R}_+\Big).
\end{equation}

We decompose $\mass_t$ and $\moment_x$ into
the sum of an absolutely continuous measure and a singular measure. We denote
\[
\mass_t=:\mass_t^a+\mass_t^s, \qquad
\moment_x=:\moment_x^a+\moment_x^s,
\]
where $\mass_t^a$ and $\moment_x^a$ are absolutely continuous
measures, $\mass_t^s$ and $\moment_x^s$ are singular measures.
More precisely, if $(x,t)$ is a jumping
point \big(that is, $y(x-,t)<y(x,t)$\big), then $\mass_t$ and $\moment_x$ are
singular at $(x,t)$; if $y(x-,t)=y(x,t)$, then $\mass$ and $\moment$ are continuous,
hence $\mass_t$ and $\moment_x$ are not singular.
The detailed proof of (\ref{MPG3smu}) is carried out  by verifying that $\mass_t^a=-\moment_x^a$
and $\mass_t^s=-\moment_x^s$ in steps \#1(a) and \#1(b) below.

\noindent {\bf Step \#1(a)} [the regular part of the mass equation]. We show that $\mass_t^a=-\moment_x^a$. If $(x,t)$ is a jumping point, then
$\mass_t^a=-\moment_x^a=0$. Otherwise,  $y(x,t)$ is continuous at $(x,t)$, and by
the definition of $\mass$ and $\moment$, we have
\begin{equation}\label{EPMtsmu}
\mass_t^a=\rho_0(y)\displaystyle\frac{\partial y}{\partial t}, \ \text{and} \
\moment_x^a=\rho_0(y)\Big(u_0(y)+\kappa
E_0(y)t\Big)\displaystyle\frac{\partial y}{\partial x}.
\end{equation}
By remark \ref{yxcont}, the equality $Q_{x,t}(y(x,t))=0$ is valid in the neighborhood 
of $(x,t): \ \displaystyle Q_{x,t}(y(x,t))=tu_0(y)+y+\frac12 \kappa E_0(y)t^2-x \equiv 0$.
Taking partial derivatives  with respect to $t$ and
$x$ yields
\[
u_0(y)+\kappa E_0(y)t+\frac{\partial Q_{x,t}(y)}{\partial
y}\frac{\partial y}{\partial t}=0 \ \ \text{and} \ \ 
-1+\frac{\partial Q_{x,t}(y)}{\partial y}\frac{\partial y}{\partial
x}=0.
\]
Hence
\[
\displaystyle\frac{\partial y}{\partial t}=-\frac{u_0(y)+
\kappa E_0(y)t}{\partial_y Q_{x,t}(y)},
\qquad \displaystyle \frac{\partial y}{\partial x}=\frac{1}{\partial_y Q_{x,t}(y)},
\]
and therefore
\begin{equation}\label{ytyxasmu}
\displaystyle\frac{\partial y}{\partial t}=-\Big(u_0(y)+\kappa
E_0(y)t\Big)\frac{\partial y}{\partial x}.
\end{equation}
Combining (\ref{ytyxasmu}) and (\ref{EPMtsmu}) we
obtain $\mass_t^a(x,t)=-\moment_x^a(x,t)$. \vspace{0.1in}

\noindent {\bf Step \#1(b)} [the singular part of the mass equation]. We show that
$\mass_t^s=-\moment_x^s$. If $y(z-,t)=y(z,t)$, then $\mass_t^s=-\moment_x^s=0$ at $(z,t)$.
 Otherwise,  $(z,t)$ is a jumping point where we have
\begin{equation}\label{jumpYsmu}
\displaystyle\int_{y(z-,t)}^{y(z,t)} \Big(s+tu_0(s)+\frac12 \kappa
E_0(s)t^2-z\Big)\rho_0(s)ds=0,
\end{equation}
\begin{align}\label{jumpYc}
y(z,t)+tu_0(y(z,t))+\frac12\kappa E_0(y(z,t))t^2-z=0,\\
\label{jumpYmc} y(z-,t)+tu_0(y(z-,t))+\frac12\kappa E_0(y(z-,t))t^2-z=0.\\
\nonumber
\end{align}
Denote the trajectory of this jumping point by $C(t,z(t))$. According
to Lemma \ref{singularMeasure}, we have
\begin{equation}\label{jumpMEPsmu}
\mass_t^s(z,t) =B(z,t)\eta_C
\frac{\dot{z}(t)}{\sqrt{1+\dot{z}^2(t)}}
\end{equation} and
\begin{equation}\label{jumpPEPsmu}
\moment_x^s(z,t)=-B(z,t)u(z,t)\eta_C\frac{1}
{\sqrt{1+\dot{z}^2(t)}};
\end{equation}
where 
\[
B(z,t)=\displaystyle\int_{y(z-,t)}^{y(z,t)} \rho_0(s)ds, \qquad u(z,t)=\displaystyle\frac{\displaystyle\int_{y(z-,t)}^{y(z,t)}\rho_0(s)\Big(u_0(s)+
\kappa E_0(s)t\Big)ds}{B(z,t)}.
\]
Thus, to prove $\mass_t^s=-\moment_x^s$ at the jumping point $(z,t)$, it remains to  show that the propagation speed of this jumping point, $\dot{z}(t)=:\bar{u}(z,t)$ is actually given by $u(z,t)$. We provide the details of $\bar{u}(z,t)=u(z,t)$ below.

The location of this jumping point at time $t+\Delta t$ is $(t+\Delta t, z(t+\Delta t)$. Hence
\begin{equation}\label{EPjdt1smu}
\displaystyle\int_{y(z(t+\Delta t)-,t+\Delta t)}^{y(z(t+\Delta
t),t+\Delta t)} \Big(s+(t+\Delta t)u_0(s)+\frac12 \kappa
E_0(s)(t+\Delta t)^2-z(t+\Delta t)\Big)\rho_0(s)ds=0.
\end{equation}
Combining,
\begin{align} \nonumber
y(z(t+\Delta t),t+\Delta
t)=y(z,t)+O(\Delta t), \\
\label{EPTEyx1smu} y(z(t+\Delta t)-,t+\Delta
t)=y(z-,t)+O(\Delta t), \\
\nonumber z(t+\Delta
t)=z+\bar{u}(z,t)\Delta t+O(\Delta t^2). \nonumber
\end{align}
with (\ref{EPjdt1smu}) and (\ref{jumpYsmu}), we obtain
\begin{eqnarray*}
0 & = & \displaystyle\int_{y(z(t+\Delta t)-,t+\Delta
t)}^{y(z(t+\Delta t),t+\Delta t)} \Big(s+(t+\Delta t)u_0(s)+\frac12
\kappa E_0(s)(t+\Delta t)^2-z(t+\Delta t)\Big)\rho_0(s)ds \\
& = & \displaystyle\int_{y(z-,t)}^{y(z,t)} \Big(s+(t+\Delta
t)u_0(s)+\frac12 \kappa E_0(s)(t+\Delta t)^2-z-\bar{u}(z,t)\Delta
t+O(\Delta t ^2))\Big)\rho_0(s)ds \\
& & + \displaystyle\int_{y(z(t+\Delta t)-,t+\Delta t)}^{y(z-,t)}
\Big(s+(t+\Delta t)u_0(s)+\frac12 \kappa E_0(s)(t+\Delta
t)^2-z-\bar{u}(z,t)\Delta
t)\Big)\rho_0(s)ds  \\
& & + \displaystyle\int_{y(z,t)}^{y(z(t+\Delta t),t+\Delta t)}
\Big(s+(t+\Delta t)u_0(s)+\frac12 \kappa E_0(s)(t+\Delta
t)^2-z-\bar{u}(z,t)\Delta
t)\Big)\rho_0(s)ds.  \\
\end{eqnarray*}
For every $s\in \Big(y(z(t+\Delta t)-,t+\Delta t), y(z-,t)\Big)$,
using Taylor expansion and (\ref{jumpYc}), we obtain
\[ s+tu_0(s)+\frac12 \kappa E_0(s)t^2-z=y(z-,t)+tu_0(y(z-,t))+\frac12 \kappa E_0(y(z-,t))t^2-z+O(\Delta t)=O(\Delta t).\]
Hence
\[ \displaystyle\int_{y(z(t+\Delta t)-,t+\Delta t)}^{y(z-,t)}
\Big(s+(t+\Delta t)u_0(s)+\frac12 \kappa E_0(s)(t+\Delta
t)^2-z-\bar{u}(z,t)\Delta t)\Big)\rho_0(s)ds=o(\Delta t). \]
Similarly
\[ \displaystyle\int_{y(z,t)}^{y(z(t+\Delta t),t+\Delta t)}
\Big(s+(t+\Delta t)u_0(s)+\frac12 \kappa E_0(s)(t+\Delta
t)^2-z-\bar{u}(z,t)\Delta t)\Big)\rho_0(s)ds=o(\Delta t). \]
Therefore
\begin{eqnarray*}
0 & = & \displaystyle\int_{y(z(t+\Delta t)-,t+\Delta
t)}^{y(z(t+\Delta t),t+\Delta t)} \Big(s+(t+\Delta t)u_0(s)+\frac12
\kappa E_0(s)(t+\Delta t)^2-z(t+\Delta t)\Big)\rho_0(s)ds \\
& = & \displaystyle\int_{y(z-,t)}^{y(z,t)} \Big(s+(t+\Delta
t)u_0(s)+\frac12 \kappa E_0(s)(t+\Delta t)^2-z-\bar{u}(z,t)\Delta
t)\Big)\rho_0(s)ds + o(\Delta t)\\
& = & \displaystyle\int_{y(z-,t)}^{y(z,t)} \Big(s+t u_0(s)+\frac12
\kappa E_0(s)t^2-z)\Big)\rho_0(s)ds  \qquad \qquad \textrm{( --- which is 0 by (\ref{jumpYsmu}))}\\
& & + \displaystyle\int_{y(z-,t)}^{y(z,t)} \Big(u_0(s)\Delta
t+\frac12 \kappa E_0(s)(2t\Delta t+\Delta t^2)-\bar{u}(z,t)\Delta
t)\Big)\rho_0(s)ds + o(\Delta t) \\
& = & \Delta t\displaystyle\int_{y(z-,t)}^{y(z,t)}
\Big(u_0(s)+\kappa E_0(s)t-\bar{u}(z,t)\Big)\rho_0(s)ds +
o(\Delta t). \\
\end{eqnarray*}
This concludes our argument that $\bar{u}(z,t)=u(z,t)$,
\[
\bar{u}(z,t)=\displaystyle\frac{\displaystyle\int_{y(z-,t)}^{y(z,t)}\rho_0(s)\Big(u_0(s)+
\kappa E_0(s)t\Big)ds}
{\displaystyle\int_{y(z-,t)}^{y(z,t)}\rho_0(s)ds}=u(z,t),
\]
and $\mass_t^s=-\moment_x^s$ follows. Thus, (\ref{MPG3smu}) holds, and the
mass equation (\ref{EP-mass}) is
satisfied in the weak sense (\ref{eq:wmasssmu}).
\newline
\newline
\noindent {\bf Step \#2} [the momentum equation]. Next, we verify the momentum equation (\ref{EP-moment}) in a similar
way. We apply test functions $\phi\in
C_c^{\infty}\Big([0,\infty)\times(-\infty,\infty)\Big)$ to
(\ref{EP-moment}), then
\begin{eqnarray}\label{EPV11bsmu}
\lefteqn{\displaystyle\int_{0}^{\infty}\int_{-\infty}^{\infty} (\rho
u)_t \phi+(\rho u^2)_x\phi-\kappa\rho E \phi dxdt} \\ \nonumber & &=
\displaystyle\int_{0}^{\infty}\int_{-\infty}^{\infty} \moment_{xt}
\phi+W_{xx}\phi-Z_{xx}\phi dxdt =
\displaystyle\int_{0}^{\infty}\int_{-\infty}^{\infty}
\moment\phi_{xt} +(W-Z)\phi_{xx} dxdt,
\end{eqnarray}
where
\[
W(x,t):=\int_0^x \displaystyle\frac{\moment^2_s(s,t)}{ \mass_s(s,t)} ds,
\quad Z_{xx}(x,t)=\kappa\rho E.
\]

\noindent
We will show below, in steps \#2(a)-2(b), that
\begin{equation}\label{MtWZx38}
\frac{\partial \moment}{\partial t}=-\frac{\partial
(W-Z)}{\partial x}.
\end{equation}
This yields the existence of $\Psi$ such that
\begin{equation}\label{EPC11bsmu}
\moment=\displaystyle\frac{\partial \Psi}{\partial x}  \qquad
W-Z=-\displaystyle\frac{\partial \Psi}{\partial t},
\end{equation}
which in turn, implies that the momentum equation, (\ref{EP-moment}), holds in its weak formulation (\ref{EPV11bsmu}).

\noindent {\bf Step \#2(a)}.
The main claim here is that $Z_x$ is given by $\int_0^{y(x,t)} \kappa\rho_0(s)E_0(s)ds$, that is, if we let $F$ denote $Z_x$ then
\begin{equation}\label{FormulaFsmu}
F_{x}(x,t)=\kappa\rho(x,t)E(x,t), \qquad  F(x,t):=\int_0^{y(x,t)} \kappa\rho_0(s)E_0(s)ds.
\end{equation}
As before, we distinguish between two cases. In the case $x$ is a continuity point,
$y(x-,t)=y(x,t)$, then
\begin{subequations}
\begin{equation}\label{FxCsmu}
\displaystyle\frac{\partial F}{\partial
x}=\kappa\rho_0(y)E_0(y)\frac{\partial y}{\partial x}.
\end{equation}
Combining $\rho(x,t)=\disp\frac{\partial}{\partial x} \mass(x,t)$ and
(\ref{WSEPsmu}), we obtain
\begin{equation}\label{FxrhoCsmu}
\rho_0(y)\frac{\partial y}{\partial x}=\frac{\partial }{\partial
x}\int_0^{y(x,t)} \rho_0(s)ds=\frac{\partial \mass(x,t)}{\partial
x}=\rho(x,t),
\end{equation}
and
\begin{equation}\label{FxmCsmu}
\left.
\begin{array}{lll} E(x,t) & = & \displaystyle \int_{-\infty}^{x}\rho(w,t)dw=\int_{-\infty}^{x}
\displaystyle\frac{\partial}{\partial w}\Big( \int_{0}^{y(w,t)}\rho_0(s)ds\Big) dw \vspace{0.05in} \\
& = & \displaystyle\int_{0}^{y(x,t)}\rho_0(s)ds -
\int_{0}^{-\infty}\rho_0(s)ds=\int_{-\infty}^{y(x,t)}\rho_0(s)ds=E_0(y).\\
\end{array} \right.
\end{equation}
Plugging (\ref{FxrhoCsmu}) and (\ref{FxmCsmu}) into (\ref{FxCsmu}),
we obtain (\ref{FormulaFsmu}), $\displaystyle F_x=\kappa\rho(x,t)E(x,t)$.
\end{subequations}

Next, we consider the case of a jump discontinuity at $x=z$, where $y(z,t)>y(z-,t)$. Then
\[
\left. \begin{array}{lll}  \displaystyle\frac{\partial F}{\partial
x}\bigg|_{x=z} & = & \displaystyle\frac{\partial}{\partial x}
\int_{-\infty}^{y(x,t)}\kappa\rho_0(s)E_0(s)ds\bigg|_{x=z}
 =  \displaystyle \delta(x-z)
\int_{y(z-,t)}^{y(z,t)}\kappa\rho_0(s)E_0(s)ds \vspace{0.05in}
\\ & = & \displaystyle
\delta(x-z) \int_{y(z-,t)}^{y(z,t)}\kappa E_0(s)dE_0(s)
 =  \displaystyle
\delta(x-z) \frac{\kappa}{2}\bigg( E_0(y(z,t))^2-E_0(y(z-,t))^2\bigg). \\
\end{array} \right.
\]
On the other hand
\[
\left. \begin{array}{lll} k\rho(z,t)E(z,t) & = &
\kappa\displaystyle\frac{\int_0^{y(x,t)}\rho_0(s)ds}{\partial
x}\bigg|_{x=z} \frac12\bigg(\displaystyle\int_{-\infty}^{z-}
\rho(s,t)ds+\int_{-\infty}^{z+}
\rho(s,t)ds\bigg) \vspace{0.05in} \\
& = & \displaystyle\frac{\kappa}{2} \displaystyle \delta(x-z)
\int_{y(z-,t)}^{y(z,t)}\rho_0(s)ds
\bigg(\displaystyle\int_{-\infty}^{z-}
\frac{\partial R(s,t)}{\partial s}ds+\int_{\infty}^{z}\frac{\partial R(s,t)}{\partial s}ds\bigg)\vspace{0.05in} \\
& = & \displaystyle\frac{\kappa}{2} \displaystyle \delta(x-z)
\int_{y(z-,t)}^{y(z,t)}\rho_0(s)ds
\bigg(\displaystyle\int_{-\infty}^{y(z-,t)}
\rho_0(s)ds+\int_{-\infty}^{y(z,t)}\rho_0(s)ds\bigg)\vspace{0.05in} \\
& = & \displaystyle
 \frac{\kappa}{2}\delta(x-z) \bigg(
E_0(y(z,t))-E_0(y(z-,t))\bigg)\bigg(E_0(y(z-,t))+ E_0(y(z,t))
\bigg) \vspace{0.05in} \\
& = & \displaystyle 
\delta(x-z) \frac{\kappa}{2}\bigg(
E_0(y(z,t))^2-E_0(y(z-,t))^2\bigg).
\end{array} \right.
\]
Thus, recover (\ref{FormulaFsmu}) $\displaystyle F_x=\kappa\rho(x,t)E(x,t)$ also at any jumping point $z$,
\begin{equation}\label{ZxF33}
Z_x=F=\int_0^{y(x,t)} \kappa\rho_0(s)E_0(s)ds.
\end{equation}

\noindent {\bf Step \#2(b)}. To prove (\ref{MtWZx38}), we decompose $\moment_t$ and $(W-Z)_x$ into the sum of
an absolutely continuous measure  and a singular measure, denoting
\[
\moment_t=\moment_t^a+\moment_t^s, \qquad (W-Z)_x=(W-Z)_x^a+(W-Z)_x^s,
\]
where $\{\}^a$ denotes an absolutely continuous measure and $\{\}^s$
denotes a singular measure. The absolutely continuous measure of $(W-Z)_x$ is given by
$(\widetilde{W}_x-F)$, where
\begin{equation}\label{tildeW}
\widetilde{W}:=\int_0^{y(x,t)}
\rho_0(s)\Big(u_0(s)+E_0(s)t\Big)^2 ds.
\end{equation}
 Since $y(x,t)$ is
continuous almost everywhere, we have $W_x=\widetilde{W}_x$ almost
everywhere.

Consider the absolutely continuous parts: if $y(x-,t)<y(x,t)$, then
$\moment_t^a(x,t)=-(W-Z)_x^a(x,t)=0$; otherwise, if $y(x-,t)=y(x,t)$, then
\begin{subequations}
\begin{equation}\label{abPtsmu} \left.
\begin{array}{lll} \moment_t^a(x,t) &
= & \rho_0(y)\Big(u_0(y)+\kappa
E_0(y)t\Big)\displaystyle\frac{\partial
y}{\partial t}+\displaystyle\int_0^{y(x,t)}\kappa \rho_0(s)E_0(s)ds \vspace{0.1in} \\
& = &\rho_0(y)\Big(u_0(y)+\kappa
E_0(y)t\Big)\displaystyle\frac{\partial
y}{\partial t}+F, \\
\end{array} \right.
\end{equation} and
\begin{equation}\label{abWQxsmu} (W-Z)_x^a(x,t)=\widetilde{W}_x^a-F=\rho_0(y)\Big(u_0(y)+\kappa
E_0(y)t\Big)^2\displaystyle\frac{\partial y}{\partial x}-F.
\end{equation}
\end{subequations}
Combining (\ref{abPtsmu}), (\ref{abWQxsmu}) and (\ref{ytyxasmu}), we
obtain $\moment_t^a=-(W-Z)_x^a$.\newline
For the singular parts: if $y(x-,t)=y(x,t)$, then $\moment_t^s(x,t)=-(W-Z)_x^s(x,t)=0$; otherwise, at any jumping point $(z,t)$ we have,
(by the definition of $\mass$, $\moment$ and $Z$),
\[
\moment_t^s(z,t)=B(z,t)
u(z,t)\eta_C\frac{\dot{z}(t)}{\sqrt{1+\dot{z}(t)^2}},
\]
\[
(W-Z)_x^s(z,t)=W_x^s(z,t)= -B(z,t)
u^2(x,t)\eta_C\frac{1}{\sqrt{1+\dot{z}(t)^2}}.
\]
When deriving $(W-Z)_x^s$, we used the fact that (i) $Z_x^s \equiv 0$ (since $Z_x=F$ is bounded everywhere), and (ii)  we have already shown, $\dot{z}=u(z,t)$.
Hence, we also have for the singular part $\moment_t^s=-(W-Z)_x^s$.
Thus, the momentum equation (\ref{EP-moment}) holds in the weak
sense.
\end{proof}

The proof of theorem \ref{EP1} covers of course the case of smooth solutions and in particular, the variational formula \eqref{WSEPsmu} describes  the globally-in-time smooth solutions in the sub-critical case \cite{EnLT01}. In this sense,  (\ref{WSEPsmu}) could be viewed  as an extension that covers both  sub-critical and
 super-critical initial configurations.
We close this section by reproducing
the proof of theorem \ref{IntroThm:1Ds} for the simpler case of smooth solutions.

\begin{lemma}\label{weakeqsm}
Consider (\ref{EP1}) with a smooth solution on $[0,T)$. Then $(\rho,\rho u)$ given in
(\ref{WSEPsmu}) is that smooth solution.
\end{lemma}

\begin{proof} We denote the smooth solution by $(\widetilde{\rho},\widetilde{u})$, the solution given by
(\ref{WSEPsmu}) by $(\rho,u)$, and we show
$(\widetilde{\rho}(\cdot,t),\widetilde{u}(\cdot,t))=(\rho(\cdot,t),u(\cdot,t))$,
$t\in[0,T)$. We can solve the equation by the method of
characteristics, and we have shown the details in the proof of Lemma
\ref{lemma:TCG}. Since the solution is smooth, no characteristics
will intersect. For every given $x\in \mathbb{R}$, $t\in[0,T)$,
there is an unique $\widetilde{y}(x,t)$, such that the
characteristic which emanates from $\widetilde{y}(x,t)$ arrives $x$
at time $t$. We claim $\widetilde{y}(x,t)$ is the unique minimizer
of $\int_0^y Q_{x,t}(s)\rho_0(s)ds$, that is,

\[
\widetilde{y}(x,t)= \arginf_y \int_0^y Q_{x,t}(s)\rho_0(s)ds.
\]
Otherwise there exists $w\neq \widetilde{y}$ such that
\[ \displaystyle\int_0^{\widetilde{y}(x,t)} Q_{x,t}(s)\rho_0(s)ds\geq \int_0^w Q_{x,t}(s)\rho_0(s)ds, \]
i.e., \[
\displaystyle\int_{w}^{\widetilde{y}(x,t)}Q_{x,t}(s)\rho_0(s)ds\geq
0, \quad \textrm{if}\,\, w<\widetilde{y}(x,t),
\]
\[
\displaystyle\int_{\widetilde{y}(x,t)}^{w}Q_{x,t}(s)\rho_0(s)ds\leq
0, \quad \textrm{if}\,\, w>\widetilde{y}(x,t).
\]
Combining the above inequalities with $Q_{x,t}(\widetilde{y})=0$, we
obtain: if $w<\widetilde{y}$, then $\exists z\in (w,\widetilde{y})$,
such that $Q_{x,t}(z)\geq 0$, which implies the characteristics emanate form
$z$ and $\widetilde{y}$ must intersect no latter than $t$; otherwise 
$w>\widetilde{y}$, then $\exists z\in (\widetilde{y},w)$, such that
$Q_{x,t}(z)\leq 0$, which implies the characteristics emanate form
$z$ and $\widetilde{y}$ must intersect no later than $t$. This is a
contradiction to smooth solution exists up to $T>t$. Therefore,
\[
\mass(x,t):=\int_0^{\widetilde{y}(x,t)} \rho_0(s) ds, \qquad
 \moment(x,t):=\int_0^{\widetilde{y}(x,t)} \Big(u_0(s)+\kappa
E_0(s)t\Big)\rho_0(s) ds,
\]
and
\[
\rho(x,t)=\frac{\partial \widetilde{y}(x,t)}{\partial x}
\rho_0(\widetilde{y}), \qquad
 \rho(x,t)u(x,t)=\frac{\partial \widetilde{y}(x,t)}{\partial x}
\rho_0(\widetilde{y})\Big(u_0(\widetilde{y})+\kappa
E_0(\widetilde{y})t\Big).
\]
Thus
\[
u(x,t)=u_0(\widetilde{y})+\kappa
E_0(\widetilde{y})t=\widetilde{u}(x,t).
\]
To verify $\rho(x,t)=\widetilde{\rho}(x,t)$, it is enough show that
\begin{equation}\label{IM} \int_{-\infty}^{x} \widetilde{\rho}(s,t)
ds= \int_{-\infty}^{\widetilde{y}(x,t)} \rho_0(s)
ds\bigg(=\mass(x,t)=\int_{-\infty}^{x} \rho(s,t)ds\bigg), \quad
\forall x,\quad \forall t<T.
\end{equation}
We denote the inverse function of $\widetilde{y}$ by
$\widetilde{x}(y,t)$, which means the characteristic starts from $y$
arrives $\widetilde{x}$ at time $t$. Then (\ref{IM}) is equivalent
to
\begin{equation}\label{IM2}
\int_{-\infty}^{\widetilde{x}(y,t)} \widetilde{\rho}(s,t) ds=
\int_{-\infty}^{y} \rho_0(s) ds, \quad \forall y, \quad \forall t<T.
\end{equation}
Notice that $y=\widetilde{x}(y,0)$, so (\ref{IM2}) equivalents to:
for every fixed $y$, $E(y,t):=\int_{-\infty}^{\widetilde{x}(y,t)}
\widetilde{\rho}(s,t) ds$ is a constant. Physically, this is clear:
since no charge can across the particle path, by conservation of
charge, $E$ remains constant along $\widetilde{x}(y,t)$. To show $E$
is a constant along $\tilde{x}(y,t)$, we take the time derivative of
$E$, then
\[ \left. \begin{array}{lll}
\displaystyle\frac{\partial}{\partial t} E(y,t) \vspace{.1in} & = &
\displaystyle\frac{d
\widetilde{x}(y,t)}{dt}\widetilde{\rho}(\widetilde{x},t)+\int_{-\infty}^{\widetilde{x}(y,t)}\frac{\partial
\widetilde{\rho}(s,t)}{\partial
t}ds \vspace{.1in} \\
& = &
\widetilde{u}(\widetilde{x},t)\widetilde{\rho}(\widetilde{x},t)-
\displaystyle\int_{-\infty}^{\widetilde{x}(y,t)}
\bigg(\widetilde{\rho}(s,t)\widetilde{u}(\widetilde{x},t)\bigg)_x ds
=\widetilde{u}\widetilde{\rho}-\widetilde{\rho}\widetilde{u} =0. \\
\end{array} \right.
\]
\end{proof}

\section{Weak solutions of the repulsive pressureless Euler-Poisson system} \label{RPEP}

\setcounter{figure}{0} \setcounter{equation}{0}

For $\kappa>0$, (\ref{WSEPsmu}) still yields a weak solution of the
Euler-Poisson system, although it  may be different from the one corresponding
to the ``sticky particle model". The following
example demonstrates this point.

\begin{example} Let $\kappa=1$. We consider a system of two initial Dirac masses: both
have mass 1, their initial positions are 0 and 1, their initial
velocities are 2 and 0, respectively. Hence the initial density and
momentum are $ \rho_0(x)=\delta(x)+\delta(x-1)$, and
$\rho_0(x)u_0(x)=2\delta(x)$. The characteristics of them are
\[
y_1(t)=2t+\frac14 t^2, \qquad y_2(t)=1+\frac34 t^2.
\]
If we set $y_1(t)=y_2(t)$, then it yields two solutions:
$t_1=2-\sqrt{2}$ and $t_2=2+\sqrt{2}$. The dynamic of the weak
solution given by our formula is: the two particles collides at time
$t_1$, then they stick as one particle, when this new particle
arrives the location $y_1(t_2)=y_2(t_2)$ at time $t_2$, it splits
into two particles again.
\end{example}

Therefore, in the repulsive case $\kappa>0$, the weak solution given by
(\ref{WSEPsmu}) subject to smooth initial data --- even if
the data is super-critical,  will eventually become smooth again; consult remark \ref{rem:xyz} below.

To further clarify the different behavior of  solutions to the \emph{attractive} and \emph{repulsive}
pressureless Euler-Poisson, consider a two particles Euler-Poisson system with mass
$m_1,m_2$, initial velocity $u_1,u_2$, initial position $y_1<y_2$.
If $\kappa< 0$, then
\begin{equation}\label{y1y2e}
y_1(t)=y_1+u_1t+\frac12\kappa\frac{m_1}{2}t^2=y_2(t)=y_2+u_2t+\frac12\kappa(m_1+\frac{m_2}{2})t^2
\end{equation}
always yields one and only one positive solution $t_c$, such that
\begin{eqnarray*}
  y_1(t)<y_2(t), &  0\leq t<t_c &\\
 y_1(t)>y_2(t), & t>t_c.
\end{eqnarray*}
If $\kappa>0$, however, then (\ref{y1y2e}) can  either have no positive solution
or \emph{two} positive solutions, and for this is the reason, our representation
formula gives a ``non-sticky particle" solution for the repulsive model.

We now turn to the proof of theorem \ref{IntroThm:1Ds} in the repulsive case, $\kappa>0$.
\begin{proof} To verify the mass equation
(\ref{EP-mass}), we decompose $\mass_t$ and $\moment_x$ into two
parts: an absolutely continuous measure and a singular measure.
We distinguish between two cases.

(1). If $y(x-,t)=y(x,t)$, one can trace $x$ backward along the
characteristic at least for a short time, and then applies the argument
of Step \#1(a) in the proof of the attractive case, $\kappa\leq 0$ case.

(2) If $y(x-,t)<y(x,t)$, then there is a $\delta$-shock. The dynamics of the
$\delta$-shock for $\kappa>0$ is a little different from the one for
$\kappa\leq0$ in that  the shock may split
or disappear. But before it splits or disappears, the argument of
Step \#1(b) in the proof of $\kappa\leq0$ case is still valid. If the
$\delta$-shock splits, we apply the argument to each sub-shock, the
conclusion still holds. 

In a similar way, we can verify the moment
equation for $\kappa>0$.
\end{proof}

\begin{nota}
Using the same technique, we can easily extend Theorem
\ref{IntroThm:1Ds} to measure density initial data, that is, Theorem
\ref{IntroThm:1Dg}.
\end{nota}

\begin{nota}\label{rem:xyz}
There is a global smooth solution if and only if $x(\alpha,t)=\alpha+u_0(\alpha)t+\frac12 \kappa
E_0(\alpha)t^2$ remains a monotonically increasing function with
respect to $\alpha$, that
is, the solution remains smooth as long as
\begin{equation}\label{CTrepulsive}
\frac{\partial x(\alpha,t)}{\partial
\alpha}=1+u_0'(\alpha)t+\frac12\kappa \rho_0(\alpha)t^2>0, \quad
\forall \alpha\in \mathbb{R}, t>0.
\end{equation}
Solving (\ref{CTrepulsive}),
we obtain the following \emph{critical }condition: the repulsive
Euler-Poisson system admits a global smooth solution if and only if
$u_0'>-\sqrt{2\kappa\rho_0}$, \cite{EnLT01}. If the initial configuration fails to satisfy this  critical threshold condition, then there are two critical times
associated with the repulsive Euler-Poisson system. At the first
critical time
\[
t_{c_1}=\inf\{t| \exists \alpha\in \mathbb{R}, \,\,
\textrm{such that} \,\, 1+u_0'(\alpha)t+\frac12\kappa
\rho_0(\alpha)t^2=0\},
\]
the solution loses its $C^1$ smoothness; but
after the second critical time,
\[
t_{c_2}=\inf\{t|
t>t_{c_1},\,1+u_0'(\alpha)t+\frac12\kappa \rho_0(\alpha)t^2>0, \,
\forall \alpha\in \mathbb{R}\},
\]
the solution given by
(\ref{WSEPsmu}) becomes $C^1$  again.
\end{nota}

\begin{nota}
The $\delta$-shock splits for $\kappa>0$ because the particles have
``memory" back to time $t=0$. More precisely, for $\kappa>0$, we have the
following two constructions  which may yield different solutions at time $t$:

(1) Use (\ref{WSEPsmu}) with initial data $(\rho_0,u_0)$ to compute $\rho(\cdot,t)$ and $\rho
u(\cdot,t)$ directly.

(2) Fix $0<t_1<t$. First use (\ref{WSEPsmu}) with initial data $(\rho_0,u_0)$ to construct
$\rho(\cdot,t_1)$ and $\rho u(\cdot,t_1)$. Then, solve the Euler-Poisson system subject to initial data
$\rho(\cdot,t_1)$, $\rho u(\cdot,t_1)$ to obtain $\rho(\cdot,t)$ and $\rho u(\cdot,t)$.
That is, apply (\ref{WSEPsmu}) with initial data $(\rho,u)(\cdot,t_1)$
to obtain
$\rho(\cdot,t)$ and $\rho u(\cdot,t)$. In this second construction, particles
are losing their ``history" before $t_1$.  That is, if particles collided before $t_1$, they will ``forget" they used
to be separate and therefore stick together forever.

To enforce the particles
to lose their ``history" at every moment (hence once collision happens, the particles will
forget they were separate, and stick as one particle thereafter),
one can impose the one sided Lipschiz
condition of lemma \ref{entropy}, as an entropy condition:
to this end, (\ref{EPysmu}) should be changed into
\begin{equation}\label{EPysmuentropy}
y(x,t) :=\sup\limits_{y}\bigg\{y \, \Big| \, y = \arginf_{y \in
\mathbb{R}\setminus D(t)} \int_0^y Q_{x,t}(s)\rho_0(s)ds \bigg \},
\end{equation}
where $D(t)=\{ z\,\big| \,\exists\, x\in\mathbb{R}, \exists\,
\widetilde{t}< t,\quad s.t.\quad
y(x-,\widetilde{t})<z<y(x,\widetilde{t})\}$.
\end{nota}

\section{Weak solution of multi-dimensional system with symmetry} \label{MS}

\setcounter{figure}{0} \setcounter{equation}{0}

In this section, we extend the result to the multi-dimensional
system with symmetry with explicit formulation of global weak
solutions for the weighted Euler-Poisson system outlined in theorem
\ref{IntroThm:MEP}.

\begin{proof} of theorem \ref{IntroThm:MEP}. Since the initial data
of (\ref{MultiDEP}) are spherically symmetric, the solution of
(\ref{MultiDEP}) will remain spherically symmetric,
$\rho(\mathbf{x},t)=\rho(r,t)$,
$\mathbf{u}(\mathbf{x},t)=w(r,t)\frac{\mathbf{x}}{r}$. Plugging
these into (\ref{MultiDrho}), we obtain
\[
\rho_t+\frac{\partial (\rho w)}{\partial r}
(\frac{\mathbf{x}}{r},\frac{\mathbf{x}}{r})+\rho
w(-\frac{1}{r^2}\frac{\mathbf{x}}{r},\mathbf{x})+\rho
w\frac{n}{r}=0,
\]
that is
\begin{equation}\label{rhosrt}
\rho_t+\frac{\partial (\rho w)}{\partial r} + \frac{n-1}{r}\rho w=0.
\end{equation}
Multiplying (\ref{rhosrt}) by $r^{n-1}$, we have
\begin{equation}
(r^{n-1}\rho)_t+\frac{\partial (r^{n-1}\rho w)}{\partial r}=0.
\end{equation}
Since $\rho$ is spherically symmetric, so does $V$, i.e.,
$V(\mathbf{x},t)=V(r,t)$. Therefore, $\nabla
V=\displaystyle\frac{\partial V}{\partial r}\frac{\mathbf{x}}{r}$,
and $\Delta V=V_{rr}+V_r\displaystyle\frac{n-1}{r}=\rho$.
Multiplying $V_{rr}+V_r\displaystyle\frac{n-1}{r}=\rho$ by
$r^{n-1}$, we obtain $(r^{n-1}V_r)_r=r^{n-1}\rho$. Plugging
$\rho(r,t)$, $\mathbf{u}(r,t)$ and $V(r,t)$ into (\ref{MultiDu}), we
have
\[ \Big(\rho w \frac{\mathbf{x}}{r}\Big)_t + \nabla\cdot\Big(\rho w^2
\frac{1}{r^2}\mathbf{x}\otimes\mathbf{x}\Big)=\kappa\rho
r^{n-1}\nabla V,
\]
which is
\[ \frac{\mathbf{x}}{r}(\rho w)_t+\frac{\rho w^2}{r^2}(n+1)\mathbf{x}+
\Big(\frac{\rho
w^2}{r^2}\Big)_r\frac{1}{r}(\mathbf{x},\mathbf{x})\mathbf{x}=\kappa\rho
r^{n-1} V_r\frac{\mathbf{x}}{r}.
\]
Further simplification yields
\[ (\rho w)_t+\frac{\rho w^2}{r}(n+1)+
\Big(\frac{\rho w^2}{r^2}\Big)_r r^2=(\rho w)_t+\frac{\rho
w^2}{r}(n-1)+ (\rho w^2)_r  =\kappa\rho r^{n-1} V_r. \] Multiplying
the above equation by $r^{n-1}$, we obtain
\begin{equation}\label{usrt}
(r^{n-1}\rho w)_t+ (r^{n-1}\rho w^2)_r  =\kappa r^{n-1}\rho r^{n-1}
V_r. \end{equation} Therefore, let
$\varsigma(|\mathbf{x}|,t)=|\mathbf{x}|^{n-1}\rho(|\mathbf{x}|,t)$,
then (\ref{rhosrt}) and (\ref{usrt}) can be rewritten as
\begin{equation}\label{MultiD1Dsym}
\begin{aligned}
\varsigma_t+(\varsigma w)_r & =0, \\
(\varsigma w)_t+(\varsigma w^2)_r & =\kappa \varsigma r^{n-1}
V_r=:\kappa \varsigma E, \quad \frac{\partial E}{\partial
r}=\frac{\partial (r^{n-1}V_{r})}{\partial r}=\varsigma.
\end{aligned}
\end{equation}
Consider (\ref{MultiD1Dsym}) with symmetric initial density $\varsigma(-r,0)=\varsigma(r,0)=r^{n-1}\rho(r,0)$ and anti-symmetric initial velocity $w(-r,0)=-w(r,0)$. Let
\begin{equation}\label{EDcons}
E(r,t)=\displaystyle\frac12\Big(\int_0^{r-}\varsigma(s,t)ds+\int_0^{r+}\varsigma(s,t)ds\Big).
\end{equation} 
The difference between the electric field (\ref{EDcons}) and the previous one (which is defined by $E(s,t)=\frac12(\int_0^{r-}\varsigma(s,t)ds+\int_0^{r+}\varsigma(s,t)ds)$) is a constant $\int_{-\infty}^{0}\varsigma(s,0)ds$. 
Physically, it corresponds to a Galilean transformation. For symmetric initial data, choosing $E$ as (\ref{EDcons}) is natural and 
convenience, since in such setting the particle located at the origin will not move and $(\varsigma,w)$ will stay symmetric.  
Applying theorem \ref{IntroThm:1Dg} with the electric field (\ref{EDcons}), we obtain a weak solution for
$(\varsigma,w)$. Then we recover $\Big(\rho(\mathbf{x},t),
(\mathbf{u}(\mathbf{x},t)\Big) $ from $\Big(\varsigma(r,t),w(r,t)\Big)$, $r\geq 0$.
\end{proof}

\section{Concluding remarks} \label{Conclu}

\setcounter{figure}{0} \setcounter{equation}{0}

We have constructed a global weak solution for the 1D pressureless Euler-Poisson
system.
For the weighted multi-dimensional pressureless Euler-Poisson system
with symmetry, which is essentially a 1D system, we have constructed
a global weak solution in the same manner. The open question is: is
it possible to at least extend the method to the real
multi-dimensional Euler system and obtain a weak solution from the
physical laws directly?


\end{document}